\documentclass[]{elsarticle}

\usepackage{graphicx}
\usepackage{subfigure}
\usepackage{amsmath,amssymb,amsthm}

\providecommand{\dd}{\mathrm{d}}
\renewcommand{\Re}{\mathop{Re}}

\DeclareMathOperator{\dom}{dom}

\theoremstyle{theorem}
\newtheorem{theorem}{Theorem}
\newtheorem{proposition}[theorem]{Proposition}
\newtheorem{corollary}[theorem]{Corollary}
\theoremstyle{definition}
\newtheorem{definition}[theorem]{Definition}
\theoremstyle{remark}
\newtheorem{remark}[theorem]{Remark}

\begin{document}
\begin{frontmatter}
\title{High order splitting schemes with complex timesteps and their application in mathematical finance}
\author{Philipp D\"orsek}
\address{Departement Mathematik, ETH Z\"urich, R\"amistra\ss e 101, 8092 Z\"urich, Switzerland}
\author{Eskil Hansen}
\address{Centre for Mathematical Sciences, Lund University, P.O.~Box 118, SE-22100 Lund, Sweden}
\begin{abstract}
	High order splitting schemes with complex timesteps are applied to Kolmogorov backward equations stemming from stochastic differential equations in Stra\-to\-no\-vich form.
	In the setting of weighted spaces, the necessary analyticity of the split semigroups can be easily proved.
	A numerical example from interest rate theory, the CIR2 model, is considered.
	The numerical results are robust for drift-dominated problems, and confirm our theoretical results.
\end{abstract}
\end{frontmatter}

\section{Introduction}
In mathematical finance, the pricing of derivative contracts can be reduced to the calculation of expected values under the risk-neutral measure, see, e.g., \cite{Bjork2004,Shreve2004b}.
This can be performed in different manners.
The most general approach is to use Monte Carlo methods, see \cite{KloedenPlaten1992,Glasserman2004}, where the recently introduced multilevel Monte Carlo approach from \cite{Giles2008} led to fast methods for problems with very low smoothness.

In contrast, we want to obtain approximations of the solution of the Kolmogorov backward PDE by splitting schemes.
Such methods were also applied successfully in Quasi Monte Carlo simulations in \cite{NinomiyaVictoir2008}, but all such approaches are in the end limited by the accuracy of the integration scheme.
Furthermore, it is not straightforward to evaluate the stochastic processes at complex timesteps, which is necessary if splittings of order higher than 2 are to be used.

We therefore solve the PDE by finite element methods in space and a high order splitting method in time.
Such high order splitting methods, overcoming the order barrier of 2 commonly encountered for splittings with nonnegative times, see \cite{BlanesCasas2005}, were introduced in \cite{CastellaChartierDescombesVilmart2009,HansenOstermann2009b} and make use of the analyticity of the split semigroups.
To show this analyticity, we make use of function spaces endowed with weighted supremum norms, originally introduced in \cite{RoecknerSobol2006} for proving the existence of solutions to martingale problems for stochastic partial differential equations, and used for the analysis of numerical methods for stochastic partial differential equations in \cite{DoersekTeichmann2010,Doersek2012,Doersek2011phdthesis,DoersekTeichmann2011,DoersekTeichmannVeluscek2012}.
It turns out that using this framework, it is very simple to prove analyticity for semigroups stemming from stochastic differential equations in Stratonovich form, hence optimal rates of convergence follow.
In particular, these results apply to problems on unbounded domains with unbounded coefficients vanishing at the (finite) boundaries of the domain.
Such problems are usually difficult to deal with in Sobolev spaces and require the use of weighted Sobolev norms vanishing at the finite boundaries, see \cite{DaskalopoulosFeehan2011,DaskalopoulosFeehan2012} for some recent results on the Heston stochastic volatility model and a thorough discussion of references on this subject.

A particularly interesting feature of the considered numerical method is that the drift part is completely separated from the diffusion part.
This means that we can choose suitable numerical schemes for each of these parts separately.
In particular, if the first order hyperbolic drift part is solved by a method such as streamline diffusion finite elements or a discontinuous Galerkin method, we can expect that the method is robust for vanishing diffusivity.
The need for such methods, notably for applications in mathematical finance, was recently observed in \cite{MarazzinaReichmannSchwab2012}.
Other methods yielding such robustness are streamline diffusion methods, see \cite{MelenkSchwab1999,GerdesMelenkSchwabSchoetzau2001}, and discontinuous Galerkin methods, see, e.g., \cite{CockburnShu1998,HesthavenWarburton2008}.
We stress that an advantage of our method is that different, optimised solvers for the drift part and the diffusion part can be used, providing a more flexible scheme.

The paper is organised as follows.
Section~\ref{sec:faframework} recalls the definitions of weighted spaces, specialising the results from \cite{RoecknerSobol2006} and \cite{DoersekTeichmann2010,DoersekTeichmann2011} to the finite-dimensional case.
Furthermore, we show analyticity for a wide class of Markov semigroups in the setting of weighted spaces.
Next, Section~\ref{sec:hosplitting} formulates the splitting scheme and contains a convergence result.
Finally, we show numerical findings for the CIR2 model in Section~\ref{sec:cir2numerics}.
In particular, we observe robustness of the suggested method for drift-dominated problems.

\section{The functional analytic framework}
\label{sec:faframework}
Let us start off by recalling some basic facts from stochastic analysis, see, e.g., \cite{KaratzasShreve1991,Oksendal2003,Protter2004} for more details.
Fix a stochastic basis $(\Omega,\mathcal{F},\mathbb{P},(\mathcal{F}_t)_{t\ge 0})$ and a $d$-dimensional standard Brownian motion $(W^j_t)_{j=1,\ldots,d, t\ge 0}$ on it.
Let $N\in\mathbb{N}$.
Consider a stochastic differential equation in Stratonovich form on the closure $D\subset \mathbb{R}^N$ of a (not necessarily bounded) Lipschitz domain,
\begin{equation}
	\dd X(t,x) = \sum_{j=0}^{d}V_j(X(t,x))\circ\dd W^j_t,
	\quad X(0,x) = x,
\end{equation}
where $W^0_t=t$ and $\circ\dd W^0_t=\dd t$, $V_j\colon D\to\mathbb{R}^N$ are Lipschitz continuous vector fields with appropriate smoothness assumptions discussed below in more detail, and $x\in D$.
We suppose that the solution $(X(t,x))_{t\ge 0}$ is well-defined in $D$, in particular that up to some common null set $N\in\mathcal{F}$, $\mathbb{P}(N)=0$, $X(t,x)\in D$ for all $x\in D$ and $t\ge 0$.
This is justified by well-known results on stochastic flows, see, e.g., \cite[Section V.7]{Protter2004}.
For $f\in\mathrm{C}^2(D)$ with suitable growth at infinity, It\^o's lemma shows that $u(t,x):=\mathbb{E}[f(X(T-t,x))]$ satisfies the backward Kolmogorov equation
\begin{subequations}
	\label{eq:backwardkolmogorov}
	\begin{alignat}{1}
	\frac{\dd}{\dd t}u(t,x) + \mathcal{L}u(t,x) &= 0, \quad t>0, \quad x\in D, \\
	u(T,x) &= f(x), \quad x\in D.
	\end{alignat}
\end{subequations}
Here, $\mathcal{L}$ denotes the ``sum of squares'' partial differential operator, defined for $g\in\mathrm{C}^2(D)$ by
\begin{equation}
	\mathcal{L}g(x)
	:=
	V_0g(x) + \frac{1}{2}\sum_{j=1}^{d}V_j^2 g(x),
\end{equation}
with $Vg(x):=V(x)\cdot\nabla g(x)$ the directional derivative for $V\colon D\to\mathbb{R}^N$ and $g\in\mathrm{C}^1(D)$.
We split this operator into
\begin{equation}
	\mathcal{L}_0 g(x):=V_0 g(x) \quad\text{and}\quad \mathcal{L}_1 g(x):=\frac{1}{2}\sum_{j=1}^{d}V_j^2 g(x).
\end{equation}
The respective split stochastic differential equations read
\begin{alignat}{3}
	\frac{\dd}{\dd t} X^0(t,x) &= V_0(X^0(t,x)), &\quad& X^0(0,x) &= x, &\quad\text{and} \\
	\dd X^1(t,x) &= \sum_{j=1}^{d}V_j(X^1(t,x))\circ\dd W^j_t, && X^1(0,x) &= x, &
\end{alignat}
cf.~also \cite{NinomiyaVictoir2008}.

We recall the following definitions from \cite{RoecknerSobol2006} and \cite{DoersekTeichmann2010,DoersekTeichmann2011,DoersekTeichmannVeluscek2012}.
\begin{definition}
	A function $\psi\colon D\to(0,\infty)$ is called \emph{D-admissible weight function} if $\lim_{\substack{\lVert x\rVert\to\infty\\x\in D}}\psi(x)=\infty$ and $\psi$ is bounded on compact sets, where $\lVert\cdot\rVert$ denotes the Euclidean norm.
\end{definition}
Note that the definition from \cite{DoersekTeichmannVeluscek2012} simplifies in this case, as $D$ is locally compact.
While the results in \cite{RoecknerSobol2006} and \cite{DoersekTeichmann2010,DoersekTeichmannVeluscek2012} are stated for real-valued functions only, they also hold true for the complex-valued versions of the spaces considered here.
\begin{definition}
	Fix $k\in\mathbb{N}_0$.
	For $j=0,\dots,k$, let $\psi_j\colon D\to(0,\infty)$ be D-admissible weight functions.
	The space $\mathcal{B}^{(\psi_j)_{j=0,\dots,k}}_k(D)$ is defined as the closure of $\mathrm{C}_b^{k}(D)$, the space of functions $f\colon D\to\mathbb{C}$ such that $f$ is bounded and $k$ times differentiable with all derivatives up to order $k$ continuous and bounded, with respect to the norm $\lVert\cdot\rVert_{(\psi_j)_{j=0,\ldots,k},k}$, where
	\begin{equation}
		\lVert f\rVert_{(\psi_j)_{j=0,\ldots,k},k}:=\lVert f\rVert_{\psi_0} + \sum_{j=1}^{k}\lvert f\rvert_{\psi_j,j}
	\end{equation}
	with
	\begin{equation}
		\lVert f\rVert_{\psi_0} := \sup_{x\in D}\psi_0(x)^{-1}\lvert f(x)\rvert
		\quad\text{and}\quad
		\lvert f\rvert_{\psi_j,j} := \sup_{x\in D}\psi_j(x)^{-1}\lVert D^jf(x)\rVert.
	\end{equation}
	Here, $D^j f(x)$ is seen as a $j$-linear form $(\mathbb{C}^N)^j\to\mathbb{C}$, endowed with the norm
	\begin{equation}
		\lVert D^j f(x)\rVert
		:=
		\sup_{\substack{z_1,\ldots,z_j\in\mathbb{C}^N\\\lVert z_1\rVert,\cdots,\lVert z_j\rVert\le 1}}\lvert D^j f(x)(z_1,\cdots,z_j)\rvert.
	\end{equation}
	We also write $\mathcal{B}^{\psi}(D):=\mathcal{B}^{\psi}_0(D)$.
\end{definition}
It is shown in \cite{Doersek2011phdthesis,DoersekTeichmannVeluscek2012} that $\mathcal{B}^{(\psi_j)_{j=0,\ldots,k}}_k(D)\subset\mathrm{C}^k(D)$, and that it is a Banach space.
Furthermore, the results of \cite{RoecknerSobol2006,DoersekTeichmann2010,Doersek2011phdthesis} yield the following generalised Feller property.
\begin{proposition}
	\label{prop:genfeller}
	Let $\psi$ be a D-admissible weight function.
	Suppose that $(P_t)_{t\ge 0}$ is a family of bounded linear operators on $\mathcal{B}^{\psi}_0(D)$ such that
	\begin{itemize}
		\item $P_0=I$, the identity on $\mathcal{B}^{\psi}(D)$,
		\item $P_{t+s}=P_tP_s$ for all $t$, $s\ge 0$, and
		\item $\lim_{t\to 0}P_t f(x)=f(x)$ for all $x\in D$ and $f\in\mathcal{B}^{\psi}(D)$.
	\end{itemize}
	Then, $(P_t)_{t\ge 0}$ is strongly continuous on $\mathcal{B}^{\psi}(D)$.
\end{proposition}

For simplicity, we formulate the following result for polynomially growing functions and vector fields with bounded derivatives.
Therefore, we fix $s\in\mathbb{N}$ large enough and set $\psi_\ell(x):=(1+\lVert x\rVert^2)^{(s-\ell)/2}$, $\ell\in\mathbb{N}_0$ and $\psi(x):=\psi_0(x)$.
Similar results can be obtained for different choices, see, e.g., \cite{DoersekTeichmann2011} for the case of exponentially growing functions and bounded vector fields.
It is straightforward to see that $P_t f(x):=\mathbb{E}[f(X(t,x))]$, $P_t^0 f(x):=\mathbb{E}[f(X^0(t,x))]$ and $P_t^1 f(x):=\mathbb{E}[f(X^1(t,x))]$ define strongly continuous semigroups on $\mathcal{B}^{\psi}(D)$, and that their generators are suitable extensions of $\mathcal{L}$, $\mathcal{L}_0$ and $\mathcal{L}_1$, respectively.
Here, we apply the usual time inversion to turn the final value problem \eqref{eq:backwardkolmogorov} into an initial value problem.
\begin{proposition}
	\label{prop:regularity}
	Let $k\in\mathbb{N}$.
	Assume that for $j=0,\dots,d$, $V_j\in\mathrm{C}^{k+1}(D;\mathbb{R}^N)$ with all derivatives of order $1$ to $k+1$ bounded.
	Then, $P_t\mathrm{C}^k_b(D)\subset\mathrm{C}^k_b(D)$ for all $t\ge 0$, and
	\begin{equation}
		\lVert P_t f\rVert_{(\psi_j)_{j=0,\dots,k}}
		\le C(t) \lVert f\rVert_{(\psi_j)_{j=0,\dots,k}}
		\quad\text{for all $f\in\mathcal{B}^{(\psi_j)_{j=0,\dots,k}}_k(D)$},
	\end{equation}
	where $C\colon[0,\infty)\to(0,\infty)$ is monotonously increasing.
	In particular, 
	\begin{equation}
		P_t\colon\mathcal{B}^{(\psi_j)_{j=0,\dots,k}}_k(D)\to\mathcal{B}^{(\psi_j)_{j=0,\dots,k}}_k(D)
	\end{equation}
	is a bounded operator with norm $\le C(t)$, $t\ge 0$.
\end{proposition}
\begin{proof}
	The first part and the norm bound is a consequence of \cite[Theorem~3.3, p.~223]{Kunita1984}, and the invariance of $\mathcal{B}^{(\psi_j)_{j=0,\dots,k}}_k(D)$ follows by a density argument.
	See also \cite[Lemma~13]{DoersekTeichmann2011} for a related result.
\end{proof}

\begin{proposition}
	\label{prop:Vgroup}
	Let $V\colon D\to\mathbb{R}^N$ be a continuously differentiable vector field with bounded derivative.
	Suppose that $V(x)\cdot n(x)\le 0$ for almost every $x\in \partial D$, where $n\colon\partial D\to\mathbb{R}^N$ is the (almost everywhere defined) outer unit normal vector to $D$.
	For $t\ge 0$, consider the semiflow $\mathrm{Fl}^V_t\colon D\to D$ of $V$, i.e.,
	\begin{equation}
		\frac{\dd}{\dd t}\mathrm{Fl}^V_t(x) = V(\mathrm{Fl}^V_t(x)), \quad \mathrm{Fl}^V_0(x) = x.
	\end{equation}

	Then,
	\begin{equation}
		P^V_t f(x):=f(\mathrm{Fl}^V_t(x)), \quad t\ge 0,
	\end{equation}
	defines a strongly continuous semigroup on $\mathcal{B}^{\psi}(D)$, and its generator is the closure of the operator $V\colon\mathcal{B}^{(\psi,\psi_1)}_1(D)\to\mathcal{B}^{\psi}(D)$.
	If $V(x)\cdot n(x)=0$ on $\partial D$, $(P^V_t)_{t\ge 0}$ can be extended to a strongly continuous group.
\end{proposition}
\begin{proof}
	The Lipschitz continuity of $V$ together with $V(x)\cdot n(x)\le 0$ for almost every $x\in\partial D$ yield that $\mathrm{Fl}^V_t$ is continuous and well-defined for $t\ge 0$, and that for some $\varepsilon>0$ there exists $C>0$ such that
	\begin{alignat}{2}
		\psi(\mathrm{Fl}^V_t(x))
		&\le C\psi(x) &\quad&\text{and}\\
		\psi_1(\mathrm{Fl}^V_t(x))
		&\le C\psi_1(x)
		&&\text{for all $t\in[0,\varepsilon)$ and $x\in D$}.
	\end{alignat}
	Hence, for arbitrary $f\in\mathcal{B}^{\psi}(D)$,
	\begin{equation}
		\sup_{x\in D}\psi(x)^{-1}\lvert f(\mathrm{Fl}^V_t(x))\rvert
		\le \sup_{x\in D}\psi(x)^{-1}\psi(\mathrm{Fl}^V_t(x)) \lVert f\rVert_{\psi,0}
		\le C\lVert f\rVert_{\psi,0},
	\end{equation}
	whence $P^V_t\colon\mathcal{B}^{\psi}(D)\to\mathcal{B}^{\psi}(D)$ is bounded for $t\ge 0$.
	As the semigroup property of $(P^V_t)_{t\in\mathbb{R}}$ is obvious, we can apply Proposition~\ref{prop:genfeller} to obtain that $(P^V_t)_{t\ge 0}$ indeed is a strongly continuous semigroup on $\mathcal{B}^{\psi}(D)$.

	Let us now determine its generator $\mathcal{G}^V$.
	We shall prove that $\mathcal{B}^{(\psi,\psi_1)}_1(D)\subset\dom\mathcal{G}^V$, and that $\mathcal{G}^V f(x)=Vf(x)$ on this space.
	Clearly, for $f\in\mathrm{C}^1_b(D)$, $\frac{\dd}{\dd t}|_{t=0}f(\mathrm{Fl}^V_t(x))=Vf(x)$, and $Vf\in\mathcal{B}^{\psi}(D)$.
	Hence, 
	\begin{equation}
		P^V_t f - f = \int_{0}^{t}P^V_s (Vf)\dd s,
	\end{equation}
	and this implies that $f\in\dom\mathcal{G}^V$ and $\mathcal{G}^Vf=Vf$.
	The Lipschitz continuity of $V$ implies $\lVert V(x)\rVert \le C(1+\lVert x\rVert^2)^{1/2}$ with some constant $C>0$, and we observe
	\begin{equation}
		\lvert Vf(x)\rvert
		\le C(1+\lVert x\rVert^2)^{1/2}\psi_1(x)\lvert f\rvert_{\psi_1,1}
		\le C\psi(x)\lvert f\rvert_{\psi_1,1}.
	\end{equation}
	It follows that $V\colon\mathcal{B}^{(\psi,\psi_1)}_1(D)\to\mathcal{B}^{\psi}(D)$ is bounded, and together with the closedness of $\mathcal{G}^V$ and Propostion~\ref{prop:closedopcontopsubspace} this implies $\mathcal{B}^{(\psi,\psi_1)}_1(D)\subset\dom\mathcal{G}^V$ and
	\begin{equation}
		\mathcal{G}^V f=Vf
		\quad\text{for all $f\in\mathcal{B}^{(\psi,\psi_1)}_1(D)$}.
	\end{equation}

	If $f\in\mathcal{B}^{(\psi,\psi_1)}_1(D)$, we see that $x\mapsto P^V_t f(x)$ is continuously differentiable, and that for $z\in\mathbb{C}^N$,
	\begin{equation}
		\lvert DP^V_t f(x)z\rvert
		=\lvert Df(\mathrm{Fl}^V_t(x))D\mathrm{Fl}^V_t(x)z\rvert
		\le \lvert f\rvert_{\psi_1,1}\psi_1(\mathrm{Fl}^V_t(x))\lvert D\mathrm{Fl}^V_t(x)z\rvert.
	\end{equation}
	Denote the Lipschitz constant of $V$ by $L$, then by Gronwall's inequality,
	\begin{equation}
		\lvert D\mathrm{Fl}^V_t(x)z\rvert 
		\le \exp(L\lvert t\rvert)\lVert z\rVert.
	\end{equation}
	Hence, for $t\in[0,\varepsilon)$, $\lVert P^V_t f\rVert_{(\psi,\psi_1),1}\le C\lVert f\rVert_{(\psi,\psi_1),1}$.
	As additionally $P^V_t(\mathrm{C}^1_b(D))\subset\mathrm{C}^1_b(D)$, we obtain that $\mathcal{B}^{(\psi,\psi_1)}_1(D)$ is a core for $\mathcal{G}^V$ by \cite[Proposition~II.1.7]{EngelNagel2000}.

	If we assume that $V(x)\cdot n(x)=0$ for $x\in\partial D$, we can apply the above results to both $V$ and $-V$, from which it follows that $(P^V_t)_{t\in\mathbb{R}}$ is a strongly continuous group.
	The proof is thus complete.
\end{proof}
\begin{corollary}
	\label{cor:squareanalytic}
	Let $V\in\mathrm{C}^3(D;\mathbb{R}^N)$ with bounded derivatives of order $1$ to $3$ such that $V(x)\cdot n(x)=0$ for almost every $x\in\partial D$.
	Then, 
	\begin{equation}
		P^{\frac{1}{2}V^2}_t f(x) := \int_{\mathbb{R}}f(\mathrm{Fl}^{V}_{\sqrt{t}y}(x))\frac{1}{\sqrt{2\pi}}\exp(-\frac{y^2}{2})\dd y
	\end{equation}
	defines an analytic semigroup of angle $\pi/2$ on $\mathcal{B}^{\psi}(D)$.
	Its generator is the closure of the operator $\frac{1}{2}V^2\colon\mathcal{B}^{(\psi,\psi_1,\psi_2)}_2(D)\to\mathcal{B}^{\psi}(D)$.
\end{corollary}
\begin{proof}
	This is a consequence of Proposition~\ref{prop:Vgroup} and \cite[Corollary~II.4.9]{EngelNagel2000}.
	That $\mathcal{B}^{(\psi,\psi_1,\psi_2)}_2(D)$ is a core is a consequence of Proposition~\ref{prop:regularity}.
\end{proof}
\begin{corollary}
	\label{cor:L0L1}
	Suppose that the vector field $V_0$ is continuously differentiable with bounded derivative and that the vector fields $V_j$, $j=1,\dots,d$ are three times continuously differentiable with bounded derivatives of order $1$ to $3$ on the Lipschitz domain $D\subset\mathbb{R}^N$, and satisfy
	\begin{equation}
		V_0(x)\cdot n(x)\le 0 \quad\text{and}\quad V_j(x)\cdot n(x)=0 \quad \text{for $j=1,\dots,d$ and $x\in\partial D$},
	\end{equation}
	where $n(x)$ denotes the (almost everywhere defined) outer unit normal vector on $D$ in $x\in\partial D$.
	Define $\mathcal{L}_0 g(x):=V_0 g(x)$ and $\mathcal{L}_1 g(x):=\frac{1}{2}\sum_{j=1}^{d}V_j^2g(x)$ for $g\in\mathcal{B}^{(\psi_j)_{j=0,1,2}}_2(D)$.
	Then, the closure of $\mathcal{L}_0$ generates the strongly continuous semigroup $(P^0_t)_{t\ge 0}$, and the closure of $\mathcal{L}_1$ generates the analytical semigroup of angle $\pi/2$ $(P^1_t)_{t\ge 0}$ on $\mathcal{B}^{\psi}(D)$.
\end{corollary}
\begin{proof}
	This follows from a combination of Proposition~\ref{prop:Vgroup}, Corollary~\ref{cor:squareanalytic}, and Proposition~\ref{prop:sumanalytic}.
\end{proof}
\begin{remark}
	We only obtained properties of the split semigroups $(P^0_t)_{t\ge0}$ and $(P^1_t)_{t\ge 0}$ in Corollary~\ref{cor:L0L1}.
	While perturbation theory arguments should allow us to derive properties of the semigroup $(P_t)_{t\ge 0}$ from these results by applying perturbation theory, see, e.g., \cite[Theorem~III.2.10]{EngelNagel2000}, this question is not central for the topic of this paper and is therefore left to further research.
\end{remark}

\section{High order splitting schemes}
\label{sec:hosplitting}
Recently, it was observed in \cite{HansenOstermann2009b,CastellaChartierDescombesVilmart2009} that complex timesteps can be used to obtain splitting schemes with order higher than $2$ for parabolic problems.
In the given setting, we have the following result.
\begin{theorem}
	Let $(\alpha_k)_{k=1,\dots,s}\subset[0,\infty)$, $(\beta_k)_{k=1,\dots,s}\subset\{z\in\mathbb{C}\colon\Re{z}\ge 0\}$ be the coefficients of a splitting scheme of formal order $p$.
	Assume the setting of Corollary~\ref{cor:L0L1}, and that additionally, $V_j$ for $j=0,\dots,d$ are $2p+3$ times continuously differentiable with all derivatives of order $1$ to $2p+3$ bounded.
	For $f\in\mathcal{B}^\psi(D)$, set
	\begin{equation}
		Q_{(\Delta t)}f:=P^0_{\alpha_1\Delta t}P^1_{\beta_1\Delta t}\dots P^0_{\alpha_s\Delta t}P^1_{\beta_s\Delta t}f.
	\end{equation}
	Then, for $f\in\mathcal{B}^{(\psi_j)_{j=0,\dots,2(p+1)}}_{2(p+1)}(D)$ and $T\ge 0$, there exists a constant $C_{T,f}>0$ independent of $n\in\mathbb{N}$ such that
	\begin{equation}
		\lVert Q_{(T/n)}^n f - P_T f \rVert_{\psi}
		\le
		C_{T,f}n^{-p}.
	\end{equation}
\end{theorem}
\begin{proof}
	By Proposition~\ref{prop:regularity}, $\mathcal{B}^{(\psi,\psi_1,\dots,\psi_{2(p+1)})}_{2(p+1)}(D)$ is invariant with respect to the semigroup $(P_t)_{t\ge 0}$.
	Using the results from Section~\ref{sec:faframework}, this is hence a straightforward consequence of \cite[Theorem~2.3]{HansenOstermann2009a}, see also \cite[Theorem~3.1]{HansenOstermann2009b} and \cite[Theorem~3.1]{CastellaChartierDescombesVilmart2009}.
\end{proof}
\begin{remark}
	\label{rem:4thordersplitting}
	A fourth order splitting as required by the above theorem was constructed in \cite[equation (5.1)]{CastellaChartierDescombesVilmart2009}.
	They have $s=5$ and 
	\begin{alignat}{1}
		\alpha_1&=0, \quad 
		\alpha_2=\alpha_3=\alpha_4=\alpha_5=\frac{1}{4}, \\
		\beta_1&=\beta_5=\frac{1}{10}-\frac{1}{30}i, \quad
		\beta_2=\beta_4=\frac{4}{15}+\frac{2}{15}i, \quad \text{and} \quad
		\beta_3=\frac{4}{15}-\frac{1}{5}i.
	\end{alignat}
	This scheme requires the evaluation of the semigroups $P^0_{\alpha_k\Delta t}$ and $P^1_{\beta_k\Delta t}$ to fourth order.
	The strategy of composition methods as applied in \cite{CastellaChartierDescombesVilmart2009} will fail if $\mathcal{L}_0$ does not generate a group, which is a consequence of \cite{BlanesCasas2005}.
	As $V_0$ usually does not satisfy $V_0(x)\cdot n(x)=0$ for $x\in\partial D$ in problems from mathematical finance with mean reversion because $V_0$ points towards the mean and hence $\mathcal{L}_0$ does not generate a group, we do not analyse this question further here.
\end{remark}

\section{Numerical example: The CIR2 model}
\label{sec:cir2numerics}
The CIR2 (Longstaff--Schwartz) model, see \cite{LongstaffSchwartz1992,BrigoMercurio2001}, is a popular model for interest rates.
It supposes that the short rate $r_t$ is given by 
\begin{equation}
	r_t=x_t+y_t
\end{equation}
with $x_t$ and $y_t$ being square-root diffusions (CIR processes) under the risk-neutral measure $\mathbb{Q}$,
\begin{subequations}
	\begin{alignat}{1}
		\dd x_t &= \theta_x(\mu_x-x_t)\dd t + \sigma_x\sqrt{x_t}\dd W^x_t, \\
		\dd y_t &= \theta_y(\mu_y-y_t)\dd t + \sigma_y\sqrt{y_t}\dd W^y_t
	\end{alignat}
\end{subequations}
where the stochastic basis is given by $(\Omega,\mathcal{F},\mathbb{Q},(\mathcal{F}_t)_{t\ge 0})$, $\theta_x$, $\mu_x$, $\sigma_x$, $\theta_y$, $\mu_y$ and $\sigma_y$ are positive constants, and $W^x_t$ and $W^y_t$ are independent Brownian motions.
From $(r_t)_{t\ge 0}$, prices of default-free zero coupon bonds can be derived via
\begin{equation}
	B(S,T) := \mathbb{E}[\exp(-\int_{S}^{T}r_t\dd t)|\mathcal{F}_S],
\end{equation}
where $\mathbb{E}[\cdot|\mathcal{F}_S]$ denotes conditional expectation with respect to $\mathcal{F}_S$ under $\mathbb{Q}$, and $0\le S\le T$.

For simplicity, we consider the approximation of the bond price itself, for which analytic formulas are available.
This allows us to determine the precise error of our numerical results.
As the bond price has smooth dependence on the initial values $x_0$ and $y_0$, we do not need to use non-uniform timesteps in the approximation.
In the approximation of non-smooth payoffs, however, this is necessary; for a construction of non-uniform time grids, see, e.g., \cite{GuiBabuska1986,Schoetzau1999,SchoetzauSchwab2000}.

After the usual time inversion, the PDE formulation of the problem reads
\begin{equation}
	\label{eq:CIR2PDE}
	\frac{\partial}{\partial t}u(t,x,y) = \mathcal{L}u(t,x,y),
	\quad u(0,x,y) = 1,
\end{equation}
with $u(T,x,y)=B(0,T)$ if we fix $x_0=x$ and $y_0=y$, where
\begin{alignat}{1}
	\mathcal{L}u(t,x,y) 
	&= \theta_x(\mu_x-x)\frac{\partial}{\partial x}u(t,x,y) + \frac{1}{2}\sigma_x^2 x\frac{\partial^2}{\partial x^2}u(t,x,y) \notag \\
	&\phantom{=}+
	\theta_y(\mu_y-y)\frac{\partial}{\partial y}u(t,x,y) + \frac{1}{2}\sigma_y^2 y\frac{\partial^2}{\partial y^2}u(t,x,y) - (x+y)u(t,x,y).
\end{alignat}
We split this equation according to
\begin{alignat}{1}
	\mathcal{L}_0 u_0(t,x,y) &= 
	\left[ \theta_x(\mu_x-x) - \frac{1}{4}\sigma_x^2 \right]\frac{\partial}{\partial x}u_0(t,x,y) \notag \\
	&\phantom{=}
	+\left[ \theta_y(\mu_y-y) - \frac{1}{4}\sigma_y^2 \right]\frac{\partial}{\partial y}u_0(t,x,y) 
	\quad\text{and} \\
	\mathcal{L}_1 u_1(t,x,y) &= 
	\frac{1}{2}\sigma_x^2 x\frac{\partial^2}{\partial x^2}u_1(t,x,y) + \frac{1}{4}\sigma_x^2\frac{\partial}{\partial x}u_1(t,x,y) \notag \\
	&\phantom{=}
	+\frac{1}{2}\sigma_y^2 y\frac{\partial^2}{\partial y^2}u_1(t,x,y) + \frac{1}{4}\sigma_y^2\frac{\partial}{\partial y}u_1(t,x,y) - (x+y)u_1(t,x,y),
\end{alignat}
which corresponds to a split into drift $\mathcal{L}_0$ and diffusion $\mathcal{L}_1$ after transformation into Stratonovich form.
In order to ensure that $x=0$ and $y=0$ are outflow boundaries for the hyperbolic problem $\frac{\partial}{\partial t}u(t,x,y)=\mathcal{L}_0u(t,x,y)$, we assume $4\theta_x\mu_x\ge \sigma_x^2$ and $4\theta_y\mu_y\ge \sigma_y^2$, which is weaker than the Feller condition, see, e.g., \cite{AndersenJaeckelKahl2010}.
While the given operator does not satisfy the assumptions of Section~\ref{sec:faframework} to prove that $\mathcal{L}_1$ generates an analytic semigroup, this can be proved directly.
Proposition~\ref{prop:ciranalytic} shows this result for the 1D case, and the proof generalises to the considered equation.
The necessary smoothness of the exact solution $u(t,x,y)$ of equation \eqref{eq:CIR2PDE} can also not be obtained directly from Proposition~\ref{prop:regularity}, but follows from \cite[Proposition~43]{AhdidaAlfonsi2010}.
\begin{figure}[htpb]
	\centering
	\subfigure[$\varepsilon=1.$]{
		\includegraphics[width=6cm]{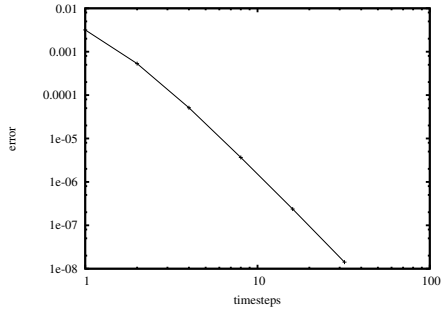}
		}
	\subfigure[$\varepsilon=.125$]{
		\includegraphics[width=6cm]{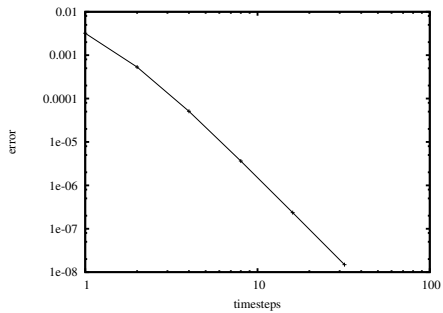}
		}
	\caption{Errors in approximation of bond prices of the CIR2 model}
	\label{fig:cir2}
\end{figure}
\begin{figure}[htpb]
	\centering
	\includegraphics[width=6cm]{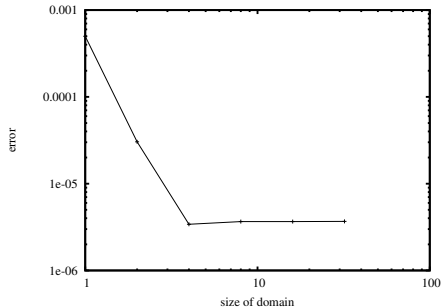}
	\caption{Effects of domain truncation}
	\label{fig:cir2-domaininc}
\end{figure}

In the computations, we fix the parameters
\begin{alignat}{2}
	\theta_x = 15.5,
	\mu_x = .025,
	\sigma_x = .2\varepsilon,
	\theta_y = 20.5,
	\mu_y = .025,
	\sigma_y = .3\varepsilon,
\end{alignat}
where we consider $\varepsilon=1.$ and $\varepsilon=.125$.
In both cases, the problem is drift-dominated.
As we want to focus on the time discretisation error, we minimise the impact of the domain truncation by cutting of far beyond any reasonable values for the state variables, at $x=16$ and $y=16$, and of the spatial discretisation by using higher order continuous finite elements in space, with $17$ grid points per direction and a polynomial degree of $4$.
This leads to a total of $4225$ spatial degrees of freedom.
The operator splitting uses the scheme from \cite[equation (5.1)]{CastellaChartierDescombesVilmart2009}, see Remark~\ref{rem:4thordersplitting} for the coefficients.
For the split equations, we employ a streamline diffusion FEM \cite{HoustonSchwabSuli2000} using a polynomial degree of $4$ for the first order hyperbolic equation involving $\mathcal{L}_0$, and a resolvent Krylov subspace method using a $30$-dimensional Krylov subspace for the approximation of the matrix exponential \cite{Guettel2010,Grimm2012} of the second order parabolic equation involving $\mathcal{L}_1$.

Figure~\ref{fig:cir2} shows the results of a numerical computation.
The plotted error is the difference between $B(0,T)$ and the numerical approximation $B_{\mathrm{app}}(0,T)$, measured in the $\mathcal{B}^{\psi}$-norm with $\psi(x,y)=(1+x^2+y^2)^{3}$.
We clearly observe 4th order convergence.
If we restrict $x+y\le 1$, which should include all values of practical interest, we see that $\psi(x,y)\le 64$.
Hence, in this region, the form of the error norm ensures that the pointwise approximation error is less than $10^{-3}$ using only $8$ timesteps.
Furthermore, the error is virtually independent of the value of $\varepsilon$, which means that our approximation is robust for the case of vanishing diffusion.

In order to understand the effect of the space truncation, we plot in Figure~\ref{fig:cir2-domaininc} the errors for different choices of the cutoff point, where the spatial mesh width is kept constant and $8$ timesteps are used.
We observe that if the domain size is increased, the error stays bounded.

\appendix
\section{Functional-analytic results}
In this section, we collect a result on closed operators and one on analytic semigroups that are easy to prove, but that were not able to find in standard literature.
\begin{proposition}
	\label{prop:closedopcontopsubspace}
	Let $(X,\lVert\cdot\rVert_X)$ and $(Y,\lVert\cdot\rVert_Y)$ be Banach spaces such that $Y$ is continuously embedded in $X$, i.e., $Y\subset X$ and $\lVert y\rVert_X\le C\lVert y\rVert_Y$ for all $y\in Y$ with some constant $C>0$, and $Z$ be dense in $Y$ (with respect to the norm of $Y$).
	Given a closed operator $A\colon\dom A\subset X\to X$ with $Z\subset\dom A$ such that $\lVert Az\rVert_X\le K\lVert z\rVert_Y$ for all $z\in Z$ with $K\ge 0$, we have $Y\subset\dom A$ and $\lVert Ay\rVert_X\le K\lVert y\rVert_Y$ for all $y\in Y$.
\end{proposition}
\begin{proof}
	By continuity, $A|_Z$ can be extended to a continuous linear operator $B\colon Y\to X$ with $\lVert By\rVert_X\le K\lVert y\rVert_Y$ and $Bz=Az$ for all $z\in Z$.
	Given $y\in Y$, let $(z_n)_{n\in\mathbb{N}}$ be such that $\lim_{n\to\infty}\lVert y-z_n\rVert_Y=0$.
	Then, by continuity of $B$, $\lim_{n\to\infty}\lVert By-Bz_n\rVert_X=0$.
	As $Bz_n=Az_n$ for all $n\in\mathbb{N}$, we observe $z_n\to y$ and $Az_n\to By$ in the norm of $X$.
	Hence, the closedness of $A$ implies $y\in\dom A$ and $Ay=By$, which yields the claim.
\end{proof}
\begin{proposition}
	\label{prop:sumanalytic}
	Fix a complex Banach space $X$.
	Given $\delta\in(0,\pi/2]$, set $\Sigma_{\delta}:=\left\{ z\in\mathbb{C}\colon\lvert\arg z\rvert<\delta \right\}$.
	Let $(A,\dom A)$, $(B,\dom B)$ be the generators of analytic semigroups $(S^A_z)_{z\in\Sigma_\delta\cup 0}$ and $(S^B_z)_{z\in\Sigma_\delta\cup 0}$ on $X$ satisfying 
	\begin{equation}
		\lVert S^A_z \rVert\le\exp(K\lvert z\rvert)
		\quad \text{and} \quad
		\lVert S^B_z \rVert\le\exp(K\rvert z\rvert)
		\quad \text{for $z\in\Sigma_\delta\cup 0$}
	\end{equation}
	for some $K>0$, and assume that the closure $(C,\dom C)$ of $(A+B,\dom A\cap\dom B)$ generates a strongly continuous semigroup $(S^C_t)_{t\ge 0}$.

	Then, $(S^C_t)_{t\ge 0}$ extends to an analytic semigroup on $\Sigma_{\delta}\cup 0$.
\end{proposition}
\begin{proof}
	We assume without loss of generality that the semigroups generated by $A$ and $B$ are bounded on $\Sigma_{\delta}$; otherwise, translate the operators appropriately first.
	By \cite[Theorem~II.4.6]{EngelNagel2000}, we have to prove that for some $\theta\in(0,\pi/2)$, $e^{\pm i\theta}C$ generate strongly continuous semigroups.
	The Hille--Yosida theorem implies existence of $\omega_0\in\mathbb{R}$ such that $\lambda-C$ is surjective for $\Re\lambda\ge\omega_0$.
	As $\Re e^{\pm i\theta}=\cos(\theta)>0$, we can find $\mu>0$ such that $\Re \mu e^{\pm i\theta}\ge\omega_0$, whence $\mu-e^{\pm i\theta}C=e^{\pm i\theta}(e^{\mp i\theta}\mu-C)$ is invertible.
	Fixing $\theta\in(0,\delta)$, we can apply the Trotter product formula \cite[Corollary~III.5.8]{EngelNagel2000} to $e^{\pm i\theta}A$ and $e^{\pm i\theta}B$ to prove that $e^{\pm i\theta}C$ generate strongly continuous semigroups.
	Repeating this argument with $A$, $B$ and $C$ replaced by $e^{i\alpha}A$, $e^{i\alpha}B$ and $e^{i\alpha}C$, where $\alpha\in(-\delta,\delta)$, we obtain the result.
\end{proof}

\section{Analyticity of the CIR semigroup in the $\mathcal{B}^{\psi}$ setting}
The following result proves the analyticity of the CIR semigroup on a $\mathcal{B}^{\psi}$-space directly, as Corollary~\ref{cor:squareanalytic} does not apply due to the lacking Lipschitz continuity of the square root.
\begin{proposition}
	\label{prop:ciranalytic}
	Set $\psi(x):=(1+x^2)^{s/2}$ with $s\in\mathbb{N}$.
	For $f\in\mathcal{B}^{\psi}([0,\infty))$, define
	\begin{equation}
		P_t f(x):=\int_{\mathbb{R}}f( (\sqrt{x}+\sigma\sqrt{t}y )^2 )\frac{1}{\sqrt{2\pi}}\exp(-\frac{y^2}{2})\dd y, \quad x\in[0,\infty),
	\end{equation}
	Then, $P_t f$ defines an analytic semigroup on $\mathcal{B}^{\psi}([0,\infty))$ and solves the PDE
	\begin{subequations}
		\begin{alignat}{2}
			\frac{\partial}{\partial t}u(t,x) &= 2\sigma^2 x\frac{\partial^2}{\partial x^2}u(t,x) + \sigma^2\frac{\partial}{\partial x}u(t,x), \quad \text{$t>0$ and $x\ge 0$}, \\
			u(0,x) &= f(x), \quad x\ge 0.
		\end{alignat}
	\end{subequations}
\end{proposition}
\begin{proof}
	It is easy to see that $(P_t)_{t\ge 0}$ defines a generalised Feller semigroup on $\mathcal{B}^{\psi}([0,\infty))$, hence it is strongly continuous by Proposition~\ref{prop:genfeller}.
Integration by parts proves
\begin{alignat*}{1}
	\frac{\partial}{\partial t}P_t f(x)
	&=\int_{\mathbb{R}}f'( (\sqrt{x}+\sigma\sqrt{t}y)^2 )\frac{\sqrt{x}+\sigma\sqrt{t}y}{\sqrt{t}}\sigma y\frac{1}{\sqrt{2\pi}}\exp(-\frac{y^2}{2})\dd y \\
	&=\frac{1}{2t}\int_{\mathbb{R}}y\frac{\partial}{\partial y}f( (\sqrt{x}+\sigma\sqrt{t}y)^2 )\frac{1}{\sqrt{2\pi}}\exp(-\frac{y^2}{2})\dd y \\
	&=\frac{1}{2t}\int_{\mathbb{R}}(y^2-1)f( (\sqrt{x}+\sigma\sqrt{t}y)^2 )\frac{1}{\sqrt{2\pi}}\exp(-\frac{y^2}{2})\dd y,
\end{alignat*}
and as the right hand side defines a bounded operator on $\mathcal{B}^{\psi}([0,\infty))$ with norm of the order $O(t^{-1})$, $(P_t)_{t\ge 0}$ is an analytic semigroup.

	To show that $P_t f$ satisfies the given PDE, we calculate
	\begin{alignat*}{1}
		\frac{\partial}{\partial x}P_t f(x) &= \int_{\mathbb{R}}f'( (\sqrt{x}+\sigma\sqrt{t}y)^2 )\frac{\sqrt{x}+\sigma\sqrt{t}y}{\sqrt{x}}\frac{1}{\sqrt{2\pi}}\exp(-\frac{y^2}{2})\dd y \\
		&= \frac{1}{2\sigma\sqrt{tx}}\int_{\mathbb{R}}\frac{\partial}{\partial y}f( (\sqrt{x}+\sigma\sqrt{t}y)^2 )\frac{1}{\sqrt{2\pi}}\exp(-\frac{y^2}{2})\dd y \\
		&= \frac{1}{2\sigma\sqrt{tx}}\int_{\mathbb{R}}yf( (\sqrt{x}+\sigma\sqrt{t}y)^2 )\frac{1}{\sqrt{2\pi}}\exp(-\frac{y^2}{2})\dd y
	\end{alignat*}
	and, similarly,
	\begin{alignat*}{1}
		\frac{\partial^2}{\partial x^2}P_t f(x) &= \frac{\partial}{\partial x}\left( \frac{1}{2\sigma\sqrt{tx}}\int_{\mathbb{R}}yf( (\sqrt{x}+\sigma\sqrt{t}y)^2 )\frac{1}{\sqrt{2\pi}}\exp(-\frac{y^2}{2})\dd y \right) \\
		&= 
		-\frac{1}{4\sigma\sqrt{tx^3}}\int_{\mathbb{R}}yf( (\sqrt{x}+\sigma\sqrt{t}y)^2 )\frac{1}{\sqrt{2\pi}}\exp(-\frac{y^2}{2})\dd y \\
		&\phantom{=}+
		\frac{1}{4\sigma^2tx}\int_{\mathbb{R}}y\frac{\partial}{\partial y}f( (\sqrt{x}+\sigma\sqrt{t}y)^2 )\frac{1}{\sqrt{2\pi}}\exp(-\frac{y^2}{2})\dd y \\
		&= 
		-\frac{1}{4\sigma\sqrt{tx^3}}\int_{\mathbb{R}}yf( (\sqrt{x}+\sigma\sqrt{t}y)^2 )\frac{1}{\sqrt{2\pi}}\exp(-\frac{y^2}{2})\dd y \\
		&\phantom{=}+
		\frac{1}{4\sigma^2tx}\int_{\mathbb{R}}(y^2-1)f( (\sqrt{x}+\sigma\sqrt{t}y)^2 )\frac{1}{\sqrt{2\pi}}\exp(-\frac{y^2}{2})\dd y.
	\end{alignat*}
	Hence,
	\begin{alignat*}{1}
		\left( 2\sigma^2x\frac{\partial^2}{\partial x^2} + \sigma^2\frac{\partial}{\partial x} \right)P_t f(x)
		&=\frac{1}{2t}\int_{\mathbb{R}}(y^2-1)f( (\sqrt{x}+\sigma\sqrt{t}y)^2 )\frac{1}{\sqrt{2\pi}}\exp(-\frac{y^2}{2})\dd y \\
		&=\frac{\partial}{\partial t}P_t f(x),
	\end{alignat*}
	which proves the claim.
\end{proof}

\section*{Acknowledgements}
The first author gratefully acknowledges support by the ETH Foundation.
The work of the second author was supported by the Swedish Research Council under grant 621-2011-5588.

\bibliographystyle{plain}
\bibliography{stochlit,mylit,mythesis}
\end{document}